\documentclass{amsart}
\usepackage{hyperref}
\usepackage{cleveref}
\newtheorem{theorem}{Theorem}[section]

\newtheorem{lemma}[theorem]{Lemma}
\newtheorem{corollary}[theorem]{Corollary}
\theoremstyle{definition}
\newtheorem{definition}[theorem]{Definition}
\newtheorem{example}[theorem]{Example}
\newtheorem{remark}[theorem]{Remark}

\newtheorem{question}[theorem]{Question}
\numberwithin{equation}{section}
\newtheorem{proposition}[theorem]{Proposition}
\begin{document}

\title[On Integer Sets Excluding Permutation Pattern Waves]{On Integer Sets Excluding \\ Permutation Pattern Waves}


\author{Kevin Cong}
\address{Harvard University, Cambridge, MA 02138} 
\curraddr{}
\email{kcong@college.harvard.edu}
\thanks{}



\date{}

\dedicatory{}

\begin{abstract}
We study Ramsey-type problems on sets avoiding sequences whose consecutive differences have a fixed relative order. For a given permutation $\pi \in S_k$, a $\pi$-wave is a sequence $x_1 < \cdots < x_{k+1}$ such that $x_{i+1} - x_i > x_{j+1} - x_j$ if and only if $\pi(i) > \pi(j)$. A subset of $[n] = \{1,\ldots,n\}$ is \emph{$\pi$-wave-free} if it does not contain any $\pi$-wave. Our first main result shows that the size of the largest $\pi$-wave-free subset of $[n]$ is $O\left((\log n)^{k-1}\right)$. We then classify all permutations for which this bound is tight. In the cases where it is not tight, we prove stronger polylogarithmic upper bounds. We then apply these bounds to a closely related coloring problem studied by Landman and Robertson.
\end{abstract}

\maketitle

\section{Introduction}

In this paper, we will study Ramsey-type problems concerning sequences of integers whose consecutive differences have a fixed relative order. Let $\pi \in S_k$ be a permutation. A \emph{$\pi$-wave} is an increasing sequence of integers whose consecutive differences have the same relative order as $\pi$. Formally, we have the following.

\begin{definition}\label{piwave}
    Let $\pi \in S_k$. A \emph{$\pi$-wave} is an increasing sequence of integers $x_1 < \cdots < x_{k+1}$ such that for all $1 \leq i, j \leq k$, we have $x_{i+1} - x_i > x_{j+1} - x_j$ if and only if $\pi(i) > \pi(j)$.
\end{definition}

We then ask the following natural Ramsey-type question: given any permutation $\pi$ and a positive integer $r$, what is the least positive integer $M$ such that one can find a monochromatic $\pi$-wave in any $r$-coloring of $[M] = \{1,\ldots, M\}$? The answer to this question is denoted by $P(\pi, r)$. 

\begin{remark}
    Note that $P(\pi, r)$ always exists by Van der Waerden's Theorem: indeed, for some sufficiently large $M$, any $r$-coloring of $[M]$ contains a monochromatic arithmetic progression of length $k^2$, which always contains a $\pi$-wave for $\pi \in S_k$.
\end{remark}

Permutation waves and the associated question of determining $P(\pi, r)$ were first introduced in 2023 by Landman and Robertson \cite{LR}, who also asked for general asymptotic bounds on $P(\pi, r)$. However, though this question for arbitrary permutation waves is recent, the question for certain special permutations dates back to around 1990, when Brown, Erdős, and Freeman \cite{BEF} and Alon and Spencer \cite{AS} studied \emph{descending waves} ($\pi$-waves for $\pi = k,k-1,\ldots,1$) with $r = 2$ colors. In fact, they considered a variant of such waves, but a simple modification of their arguments yields the following tight bound.\footnote{Precisely, the authors of \cite{AS} and \cite{BEF} worked with what we call \emph{weak-difference descending waves}. A \emph{weak-difference descending wave} of length $k+1$ is a sequence $(x_1, \ldots, x_{k+1})$ such that $x_{i+1} -x_i \geq x_{i+2} - x_{i+1}$ for all $1 \leq i \leq k-1$. They consider the same question presented in this paper but for weak-difference descending waves. See \cite{LR} for an explanation of the easy adaptation to the setting of descending waves.}

\begin{theorem}\label{BEF}\cite{AS, BEF, LR}
    Let $\pi = k, k-1, \ldots, 1$. Then $P(\pi, 2) = \Theta(k^4)$.
\end{theorem}

For a generalization of the above result to multiple colors, see \cite{LsR}. Other variants of descending waves and the corresponding Ramsey-type question have also been studied; for instance, see \cite{BBCY, CDRX}. 

\medskip

In 1995, Bollobás, Erdős, and Jin \cite{BEJ} considered the descending wave question from a different perspective. In particular, in contrast with \cite{AS} and \cite{BEF}, they fix the length of the descending permutation $\pi$ and vary the number of colors, $r$. They showed the following tight asymptotic on $P(\pi, r)$. Note that here and in the remainder of the paper, constants suppressed by big-$O$ notation are allowed to depend on $\pi$. 

\begin{theorem}\label{BEJ}\cite{BEJ}
     Let $\pi = k, k-1, \ldots, 1$. Then for all sufficiently large $r$, $$\frac{1}{2^{k!}} \cdot r(\log_2 r)^{k-1} < P(\pi, r) < \frac{1+o(1)}{k!}\cdot (r \log_2 r)^{k-1}.$$
\end{theorem}

In our work, we also study the following variant of the question of Landman and Robertson. Instead of considering a full coloring of $[P(\pi, r)]$, we ask: how many of the elements can be of a given color?

\begin{definition}
    Say a set of integers $S$ is \emph{$\pi$-wave-free} if it does not contain any $\pi$-wave. Then for any permutation $\pi \in S_k$ and positive integer $n$, we denote by $g(\pi, n)$ the largest possible cardinality of a $\pi$-wave-free set $S \subseteq [n]$. 
\end{definition}

In addition to the coloring problem, we focus on obtaining asymptotic bounds for $g(\pi, n)$, which we refer to as the \emph{density problem}. Informally, we are interested in the maximum density of a $\pi$-wave-free subset. Note that one can obtain one nontrivial upper bound on this quantity via Szemerédi's theorem: in particular, a sufficiently large subset of $[n]$ must contain an arithmetic progression of length at least $k^2$, which then contains a $\pi$-wave. However, this bound is far larger than the actual value of $g(\pi, n)$. 

Furthermore, note that any upper bound on the value of $g(\pi, n)$ implies an upper bound on the value of $P(\pi, r)$, and any lower bound on $P(\pi, r)$ implies a lower bound on $g(\pi, n)$. 

\begin{lemma}\label{gtoP}
    Let $\pi \in S_k$. 
    \begin{enumerate}
        \item If $g(\pi, n) = O\left((\log n)^t\right)$, then $P(\pi, r) = O\left(r(\log r)^t\right)$. 
        \item If $P(\pi, r) = \Omega\left(r(\log r)^t\right)$, then $g(\pi, n) = \Omega\left((\log r)^t\right)$.
    \end{enumerate}
\end{lemma}


Consequently, we focus on proving upper bounds on $g(\pi, n)$ and lower bounds on $P(\pi, r)$. In each case, our results yield bounds on both quantities. 
\medskip

For the rest of this paper, we will take $n > 1$ and $r>1$ to be integers. Our first main result is a general recursive upper bound on the size of $g(\pi, n)$. Let us first recall that for any sequence of distinct integers $\sigma = x_1, \ldots, x_t$, the \emph{normalization} $\overline{\sigma} \in S_k$ of $\sigma$ is the unique permutation such that $\overline{\sigma}(i) > \overline{\sigma}(j) $ if and only if $ x_i > x_j$. We then have the following.

\begin{theorem}\label{main}
    Let $\pi \in S_k$, $k \geq 2$. Let $\pi' \in S_{k-1}$ be the permutation obtained by removing $1$ from $\pi$ and then normalizing the result. Then $$g(\pi, n) \leq 30 (\log_2 n) \cdot g(\pi', n).$$
\end{theorem}
\begin{example}
    Let $\pi = 4312$, so $\pi' = \overline{432} = 321$. Then \Cref{main} implies that $g(4312, n) \leq 30(\log_2n) \cdot g(321, n)$. \Cref{BEJ} then shows that $g(4312, n) = O\left((\log n)^3\right)$. This bound is tight.
\end{example}

A simple inductive argument yields the following corollary, which bounds $g(\pi, n)$ based on the length of $\pi$ alone and generalize \Cref{BEJ}. 

\begin{corollary}\label{maincor}
    Let $\pi \in S_k$. 
    \begin{enumerate}
        \item We have $g(\pi, n)= O\left( (\log n)^{k-1}\right).$ 
        \item We have $P(\pi, r)= O\left(r (\log r)^{k-1}\right).$ 
    \end{enumerate}
\end{corollary}

By \Cref{BEJ}, the bound given in \Cref{main} is tight up to a constant factor when $\pi = k,k-1, \ldots,1$. In fact, tightness holds for a greater class of permutations, which we will classify precisely. A \emph{peak} of $\pi$ is an index $i$ such that $\pi(i) > \pi(i-1)$ and $\pi(i) >  \pi(i+1)$. The following bound, a natural adaptation of methods in \cite{BEF}, allows us to prove tightness for permutations when appending $k$ at either end of a permutation in $S_{k-1}$.

\begin{proposition}\label{ezconst}
    Let $\pi \in S_k$ be a permutation beginning with $k$. Let $\pi'$ be the permutation obtained by deleting $k$ from $\pi$. Suppose that $P(\pi', r) =\Omega\left( r(\log r)^s\right)$. Then $P(\pi, r) =\Omega\left(r(\log r)^{s+1}\right).$
\end{proposition}

Repeatedly using the previous result shows that \Cref{main} is tight for all permutations without peaks.

\begin{corollary}\label{peakseq}
    Let $\pi \in S_k$ be a permutation with no peaks. 
    \begin{enumerate}
        \item We have $g(\pi, n) = \Theta\left((\log n)^{k-1}\right)$.
        \item We have $P(\pi, r) = \Theta\left(r(\log r)^{k-1}\right).$
    \end{enumerate}
    \end{corollary}

However, the bound of \Cref{main} is not tight in general. In particular, a modification of \Cref{main} yields the following recursive bound, which is stronger in many cases. 

\begin{theorem}\label{stronger}
    Let $\pi \in S_k$ be a permutation such that $1$ and $2$ are not adjacent. Let $\pi'$ be the permutation obtained by removing $1$ and $2$ from $\pi$ and then normalizing. Then $$g(\pi, n) \leq 42 (\log_2 n) \cdot g(\pi', n).$$
\end{theorem}
\begin{example}
    Let $\pi = 1423$ and $\pi' = \overline{43} = 21$. Then \Cref{stronger} implies that $g(1423, n) \leq 42 (\log_2 n) \cdot g(21, n)$. \Cref{maincor} then shows that $g(1423, n) = O\left((\log n)^2\right)$. This bound is tight. 
\end{example}
Combining \Cref{peakseq} and \Cref{stronger}, we have a complete classification of all permutations $\pi$ for which \Cref{maincor} is asymptotically tight. 

\begin{theorem}\label{equality}
    Let $\pi \in S_k$. 
    \begin{enumerate}
        \item If $\pi$ has no peaks, then:
        \begin{itemize}
            \item $g(\pi, n) = \Theta\left((\log n)^{k-1}\right)$,
            \item $P(\pi, r) = \Theta\left(r(\log r)^{k-1}\right)$.
        \end{itemize}
        \item If $\pi$ has at least one peak, then 
        \begin{itemize}
            \item $g(\pi, n) = O\left((\log n)^{k-2}\right)$,
            \item $P(\pi, r) = O\left(r(\log r)^{k-2}\right)$.
        \end{itemize}
    \end{enumerate}
    
\end{theorem}

Lastly, we analyze case (2) of the above theorem and obtain tight results via \Cref{stronger} and a recursive construction for a certain class of permutations. In particular, a \emph{layered permutation} $\pi \in S_k$ is any permutation obtained from taking the descending permutation $k,k-1,\ldots,1$ and reversing some descending runs. A \emph{layer} is precisely a maximal contiguous ascending run, and a \emph{non-final layer} is any layer which is not at the end of the permutation. The \emph{size} of a layer is the number of elements in the layer.

\begin{example}
    The permutation $\pi = 789623451$ is layered. Here, there are four layers: $789$, $6$, $2345$, and $1$, with sizes $3$, $1$, $4$, and $1$, respectively. The first three of these are the non-final layers. 
\end{example}

We then have the following tight bound for all layered permutations.

\begin{theorem}\label{deccor}
    Let $\pi \in S_k$ be a layered permutation. Let $\ell$ be the number of non-final layers of $\pi$ of size at least $2$.
    \begin{enumerate}
        \item We have $g(\pi, n) = \Theta\left((\log n)^{k-\ell-1}\right)$.
        \item We have $P(\pi, r) = \Theta\left(r\log r)^{k-\ell-1}\right).$
    \end{enumerate}
\end{theorem}

The remainder of this paper is structured as follows. In \Cref{maingen}, we prove the general bounds of \Cref{main} and \Cref{maincor}. In \Cref{str}, we obtain stronger recursive bounds and classify all cases where \Cref{maincor} is tight by proving \Cref{stronger} and \Cref{equality}. In \Cref{layered}, we show the tight result of \Cref{deccor} in the case of layered permutations. In \Cref{cors}, we prove \Cref{gtoP}. Finally, in \Cref{concl}, we provide concluding remarks and suggest directions for future work. 

\textbf{Acknowledgements.} This research was conducted while the author was at the 2023 University of Minnesota Duluth REU, supported by Jane Street Capital, the National Security Agency, and the National Science Foundation (Grants 2052036 and 2140043). The author thanks Joe Gallian and Colin Defant for organizing the REU and for the opportunity to take part, as well as proofreading the paper, Noah Kravitz for helpful discussions, and Mitchell Lee and Carl Schildkraut for their comments on earlier drafts of the paper.

\section{General Upper Bounds: Proofs of \Cref{main} and \Cref{maincor}}\label{maingen}
We begin by discussing general bounds on $g(\pi, n)$ which hold for all permutations. First, we give a proof of \Cref{main}. The main idea is to use the Pigeonhole Principle repeatedly to construct a $\pi'$-wave satisfying certain conditions, which we then modify to construct a $\pi$-wave. 

\begin{proof}[Proof of \Cref{main}]
Let $\pi$, $\pi'$ be as in the theorem statement. Let $S \subseteq [n]$ be any set such that $m = |S| \geq 30(\log_2 n) \cdot g(\pi', n)$. We will show that $S$ has a $\pi$-wave. For $n \leq 30$ the statement is vacuously true as $|S| \geq 30$, hence we may assume $n > 30$. Our first goal will be to take a particular subset of $S$ from which we find a suitable $\pi'$-wave. 

Suppose that the elements of $S$ in increasing order are $x_1 < \cdots < x_m$. Partition the elements into  $1 + \lfloor \log_2 n \rfloor$ sets as follows: for each $1 \leq j \leq 1 + \lfloor \log_2 n \rfloor$, set  $$T_j = \{x_i: 1 \leq i \leq m-1, x_{i+1} - x_i \in [2^{j-1}, 2^{j})\}.$$

Note that by definition, $$\bigsqcup_{j = 1}^{1 + \lfloor \log_2 n \rfloor} T_j =S.$$ Therefore, by the Pigeonhole Principle there is some $1 \leq s \leq 1 + \lfloor \log_2 n \rfloor$ such that $$|T_s| \geq \frac{30(\log_2 n) \cdot g(\pi', n)}{ 1 + \lfloor \log_2 n \rfloor}\geq 15g(\pi', n).$$ Suppose $T_s = \{x_{i_1}, \ldots, x_{i_t}\}$ where $t \geq 15 g(\pi', n)$ and $i_1 < \cdots < i_t$. Consider the set 
$$T_s' = \left\{\left\lfloor \frac{x_{i_{2j}}}{2^{s-1}} \right\rfloor: 2 \leq j \leq \left\lfloor \frac{t}{2}\right\rfloor\right\}.$$ Note that $$x_{i_{2j+2}} - x_{i_{2j}} \geq x_{i_{2j} + 1} - x_{i_{2j}} \geq 2^{s-1},$$ hence $T_s'$ is a set of $\lfloor \frac{t}{2} - 1\rfloor$ positive integers. In particular, we have $|T_s'| \geq 6g(\pi', n)$.

Again, partition $T_s'$ into $6$ sets $$T_{s,j} = \{n\in T_s', n \equiv j \pmod 6\}$$ for $1 \leq j \leq 6$. By Pigeonhole again, at least one such set will satisfy $|T_{s,j}| \geq \frac{|T_j'|}{6} \geq g(\pi', n)$.  It follows by the definition of $g(\pi', n)$ that there is some increasing subsequence of elements of $T_{s,j}$ which forms a $\pi'$-wave. Let this sequence be $$\mathcal{S}_0 = \left(\left\lfloor \frac{x_{a_1}}{2^{s-1}} \right\rfloor,\ldots, \left\lfloor \frac{x_{a_k}}{2^{j-1}} \right\rfloor\right) = (y_1, y_2, \ldots, y_k).$$ 

\medskip

Now, our first task is accomplished: the sequence $$\mathcal{S}_1 = \left(x_{a_1}, \ldots, x_{a_k}\right) = (z_1, \ldots, z_k)$$ is a $\pi'$-wave, and moreover, the consecutive differences $d_1(i) = z_{i+1} - z_i$ are all far apart. Indeed, let $d_0(i) = y_{i+1} - y_i$ be the consecutive differences of $\mathcal{S}_0$. Since $\mathcal{S}_0$ is a $\pi'$-wave, we know that if $\pi'(i) > \pi'(j)$, then $d_0(i) > d_0(j)$. Since $x \leq \lfloor x \rfloor < x+1$, we have $$[d_0(i) - 1] \cdot 2^{s-1}  < d_1(i) < [d_0(i) + 1] \cdot 2^{s-1}.$$ Also, by construction, $d_0(i)$ is a multiple of $6$. Therefore, we have $$d_0(i) > 0 \implies d_1(i) > 5 \cdot 2^{s-1}$$ and \begin{align*} d_0(i) > d_0(j) \implies d_1(i) - d_1(j)> 4 \cdot 2^{s-1}.\end{align*} 

\medskip

Now, we construct the desired $\pi$-wave. Let $\ell = \pi^{-1}(1)$ be the index of $1$ in $\pi$. Let $u = x_{a_l+1}$. Recall that since $z_l = x_{a_l} \in T_s$, we know $u - z_l = x_{a_{l} + 1} - x_{a_l} \in [2^{s-1}, 2^s)$. Then consider the sequence $$\mathcal{S} = (z_1, \ldots, z_{l}, u, z_{l+1}, \ldots, z_k).$$ We claim that $\mathcal{S}$ is a $\pi$-wave. 

This is true as $u - z_l$ is the smallest consecutive difference, and inserting $u$ in $\mathcal{S}_1$ does not significantly affect any other consecutive differences. More precisely, suppose that $\pi(i) > \pi(j)$. Let $\mathcal{S}(i)$ be the $i$-th element of $\mathcal{S}$ and let $d(i) = \mathcal{S}(i+1) - \mathcal{S}(i)$ be the consecutive differences of $\mathcal{S}$. It suffices to show that $d(i) > d(j)$. We then have two cases.

\medskip

\noindent\textbf{Case 1: $j = \ell$.} Then if $i \leq \ell$, we have $$d(i) = d_1(i) > 5 \cdot 2^{s-1} > 2^s,$$ and if $i \geq \ell+1$, we have $$d(i) \geq d_1(i-1) - 2^s > 3 \cdot 2^{s-1} > 2^s.$$ For both possibilities, we have $d(i)> d(j)$. 

\medskip

\noindent\textbf{Case 2: $j \neq \ell$.} Because $\pi(i) > \pi(j)$, it follows that $i \neq \ell$ as well. Now, let $i' = i $ if $i \leq \ell$ and $i' = i-1$ otherwise. Similarly, let $j' = j$ if $j \leq \ell$ and $j' =  j-1$ otherwise. This is done so that for $i \neq \ell+1$, we have $$d(i) = z_{i'+1} - z_{i'} = d_1(i'),$$ and for $i = \ell + 1$ we have $$d(i) = z_{\ell+1} - u = d_1(\ell) - (u - z_\ell) > d_1(\ell)  - 2^s = d_1(i') - 2^s.$$ Hence, we have $$d_1(i') - 2^s < d(i) \leq d_1(i').$$ Analogous statements hold for $d(j)$. Therefore, we have that $$d(i)  -d(j) > d_1(i') - d_1(j') - 2^s > 4 \cdot 2^{s-1} - 2^s > 0.$$ Therefore, $d(i) > d(j)$ in this case as well. We conclude that $\mathcal{S}$ is a $\pi$-wave. 
\end{proof}

The recursive nature of the theorem now immediately implies \Cref{maincor}. This thus provides a general upper bound on $g(\pi, n)$ which is polynomial in $\log n$ and depends only on the length of $\pi$. 

\begin{proof}[Proof of \Cref{maincor}] We claim that for any $\pi \in S_k$, $g(\pi, n) \leq 30^k(\log_2 n)^{k-1}$. Indeed, induct on $k$. For $k = 1$ it is immediate that $g(\pi, n) = 2$. Now suppose the statement holds for all permutations $\pi \in S_{k-1}$. Let $\pi \in S_k$ be some permutation of length $k$, and let $\pi'$ be the permutation obtained by removing $1$ from $\pi$ and normalizing. Then $\pi' \in S_{k-1}$. By \Cref{main} and the inductive hypothesis, $g(\pi, n) < 30(\log_2 n) \cdot g(\pi', n)  < 30^k (\log_2 n)^{k-1}$. This completes the induction and the proof of the corollary. 
\end{proof}

We end this section by remarking that the above proof of \Cref{maincor} allows the implied constant to be taken as $100^k$, which is exponential in the length of $\pi$. 

\section{Stronger Bounds and Tightness in the Main Result}\label{str}

We now discuss stronger bounds and classify all cases in which \Cref{maincor} is tight. We will begin by proving \Cref{ezconst}.

\begin{proof}[Proof of \Cref{ezconst}]
    Let $\pi$, $\pi'$ be as in the theorem statement. We will derive a recursive inequality for $P(\pi, r)$. For convenience, let $\hat{P}(\pi, r) = P(\pi, r) - 1$. 
    
    Let $\mathcal{C}_0: [\hat{P}(\pi, r) ] \rightarrow [r]$ be an $r$-coloring of $[\hat{P}(\pi, r)]$ without any monochromatic $\pi$-wave. Let $\mathcal{C}_0': [\hat{P}(\pi', r)] \rightarrow [r]$ be an $r$-coloring of $[\hat{P}(\pi', r) ]$ without any monochromatic $\pi'$-wave. We will combine these colorings to form a $2r$-coloring of $[2\hat{P}(\pi, r) + \hat{P}(\pi', r)]$. 

    Indeed, let the coloring $\mathcal{C}: [2\hat{P}(\pi, r) + \hat{P}(\pi', r)] \rightarrow [2r]$ be such that
  $$
    \mathcal{C}(i) = 
    \begin{cases}
      \mathcal{C}_0(i), & \text{for }\ i \in [1, \hat{P}(\pi, r)] = L \\
      \mathcal{C}_0(i) + r, & \text{for }\ i \in (\hat{P}(\pi, r) , 2\hat{P}(\pi, r) ] = M\\
      \mathcal{C}_0'(i), & \text{for }\ i\in (2\hat{P}(\pi, r), 2\hat{P}(\pi, r) + \hat{P}(\pi', r)] = R.\
    \end{cases}
  $$
We claim that the coloring $\mathcal{C}$ has no monochromatic $\pi$-wave. Indeed, suppose that there were some monochromatic $\pi$-wave $(x_1, \ldots, x_{k+1})$. Let $c = C(x_1)$ be the color of all elements in the wave. By construction there is no monochromatic $\pi$-wave entirely in $M$. Therefore, we must have $1 \leq c \leq r$. Hence for all $i$ we have $x_i \not\in M$.

Again by construction, there is no monochromatic $\pi$-wave in $L$ or $R$. Therefore, we must have $x_1 \in L$ and $x_{k+1} \in R$. Hence, there must be some index $i$ for which $x_i \in L$ and $x_{i+1} \in R$. Since $\pi$ begins with $k$, we must have $x_2 - x_1 > x_{i+1} - x_i > \hat{P}(\pi, r)$. However, this implies that $x_2 \geq \hat{P}(\pi, r)$. Since $x_2 \not\in M$ we must have $x_2 \in R$, and hence $x_i \in R$ for $i \geq 2$. Yet $(x_2, x_3, \ldots, x_{k+1})$ forms a $\pi'$-wave. But there is no $\pi'$-wave in $R$, a contradiction. 

Therefore, we have the recursion $$\hat{P}(\pi, 2r) \geq 2\hat{P}(\pi, r) + \hat{P}(\pi', r). $$ Since $\hat{P}(\pi', r) = \Omega\left(r(\log r)^s\right)$, there are some $C, N_0$ such that for $r > N_0$ we have $\hat{P}(\pi', r) \geq Cr(\log_2 r)^s$. Hence for all $r > N_0$ we have $$\hat{P}(\pi, 2r) \geq 2\hat{P}(\pi, r) + Cr(\log_2 r)^s.$$ Solving the recursive inequality yields $\hat{P}(\pi, r) \geq C'r(\log_2 r)^{s+1}$ for some sufficiently small $C' > 0$ and all sufficiently large $r$ which are powers of $2$. Since $\hat{P}(\pi, r)$ is clearly increasing in $r$, we conclude that $\hat{P}(\pi, r) = \Omega\left(r(\log r)^{s+1}\right)$, completing the proof. 
\end{proof}

We now prove \Cref{peakseq}, therefore establishing that \Cref{maincor} is tight for permutations without peaks. First, let us note the analog of a useful remark in \cite{LR}. 

\begin{remark}[\cite{LR}]\label{reverse}
    Let $\pi$ be any permutation, and $\pi_{\mathrm{rev}}$ be the reverse of $\pi$. Then $P(\pi, r) = P(\pi_\mathrm{rev}, r)$, and $g(\pi, n) = g(\pi_\mathrm{rev}, n)$. 
\end{remark}

This follows immediately from the fact that if $S \subseteq [n]$ contains no $\pi$-wave, then the set $S_\mathrm{rev} = \{n+1-s: s \in S\}$ contains no $\pi_\mathrm{rev}$-wave. We now prove \Cref{peakseq}.

\begin{proof}[Proof of \Cref{peakseq}]
    We show item (1); item (2) then follows from \Cref{gtoP}. The upper bound is a consequence of \Cref{main}, hence it suffices to show the lower bound. Proceed with induction on the length $k$ of $\pi$. For $k = 1$ it is clear. Then suppose the statement is true for all $\pi$ of length $k-1$ and take some permutation $\pi \in S_k$ without peaks. Then note that $k$ must be at either the beginning or the end of $\pi$, as otherwise it would be a peak. By \Cref{reverse}, we can assume that $k$ is at the beginning. Then let $\pi = k,\pi'$ and note that by the inductive hypothesis, $g(\pi', n) = \Omega\left((\log_2n)^{k-2}\right)$. Hence, by \Cref{ezconst}, we have $g(\pi, n) = \Omega(\left((\log_2n)^{k-1}\right)$, completing the induction. 
\end{proof}

We now show \Cref{stronger}, a recursive bound which improves on the bound of \Cref{main} for most permutations. We use the same main idea as in the proof of \Cref{main} with a few modifications. First, we note the following simple fact.

\begin{lemma}\label{profilepick}
    Let $a_1 < a_2 < a_3$ be positive integers. Then there exists $i \in \{1,2\}$ such that $a_{i+1} - a_i \leq \frac{a_3 - a_1}{2}$.
\end{lemma}
\begin{proof}
    Simply note that if no such $i$ existed, then $$a_3 - a_1 = (a_3 - a_2) + (a_2 - a_1) > 2\left(\frac{a_3-a_1}{2}\right),$$ which is impossible.
\end{proof}
We now present the proof.

\begin{proof}[Proof of \Cref{stronger}]
Let $\pi$, $\pi'$ be as in the theorem statement. Assume that $1$ appears before $2$ in $\pi$; the other case then follows from \Cref{reverse}. Let $S$ be any set such that $m = |S| \geq 42(\log_2 n) \cdot g(\pi', n)$. We will show that $S$ has a $\pi$-wave. For $n \leq 42$ the statement is vacuously true as $|S| \geq 42$, hence we may assume $n > 42$. Our first goal will be to take a particular subset of $S$ from which we find a suitable $\pi'$-wave. 

Suppose that the elements of $S$ in increasing order are $x_1 < \cdots < x_m$. Partition the elements into $1 + \lfloor \log_2 n \rfloor$ sets as follows: for each $1 \leq j \leq 1 + \lfloor \log_2 n \rfloor$, set  $$T_j = \{x_i: 1 \leq i \leq m-1, x_{i+3} - x_i \in [2^{j-1}, 2^{j})\}.$$

Note that by definition, $$\bigsqcup_{j = 1}^{1 + \lfloor \log_2 n \rfloor} T_j =S.$$ Therefore, by Pigeonhole, there is some $1 \leq s \leq 1 + \lfloor \log_2 n \rfloor$ such that $$|T_s| \geq \frac{42(\log_2 n) \cdot g(\pi', n)}{ 1 + \lfloor \log_2 n \rfloor}\geq 21g(\pi', n).$$ Suppose $T_s = \{x_{i_1}, \ldots, x_{i_t}\}$ where $t \geq 21 g(\pi', n)$ and $i_1 < \cdots < i_t$. Consider the set 
$$T_s' = \left\{\left\lfloor \frac{x_{i_{3j}}}{2^{s-1}} \right\rfloor: 2 \leq j \leq \left\lfloor \frac{t}{3}\right\rfloor\right\}.$$ Note that $$x_{i_{3(j+1)}} - x_{i_{3j}} \geq x_{i_{3j} + 3} - x_{i_{3j}} \geq 2^{s-1},$$ hence $T_s'$ is a set of $\left\lfloor \frac{t}{3} - 1\right\rfloor$ positive integers. In particular,  we have $|T_s'| \geq 6g(\pi', n)$.

Again, partition $T_s'$ into $6$ sets $$T_{s,j} = \{n\in T_s', n \equiv j \pmod 6\}$$ for $1 \leq j \leq 6$. By Pigeonhole again, at least one such set will satisfy $|T_{s,j}| \geq \frac{|T_j'|}{6} \geq g(\pi', n)$.  It follows by the definition of $g(\pi', n)$ that there is some increasing subsequence of elements of $T_{s,j}$ which forms a $\pi'$-wave. Let this sequence be $$\mathcal{S}_0 = \left(\left\lfloor \frac{x_{a_1}}{2^{s-1}} \right\rfloor,\ldots, \left\lfloor \frac{x_{a_{k-1}}}{2^{j-1}} \right\rfloor\right) = (y_1, y_2, \ldots, y_{k-1}).$$ 

\medskip

Now, our first task is accomplished: the sequence $$\mathcal{S}_1 = \left(x_{a_1}, \ldots, x_{a_{k-1}}\right) = (z_1, \ldots, z_{k-1})$$ is a $\pi'$-wave, and moreover, the consecutive differences $d_1(i) = z_{i+1} - z_i$ are all far apart. Indeed, let $d_0(i) = y_{i+1} - y_i$ be the consecutive differences of $\mathcal{S}_0$. Since $\mathcal{S}_0$ is a $\pi'$-wave, we know that if $\pi'(i) > \pi'(j)$, then $d_0(i) > d_0(j)$. Since $x \leq \lfloor x \rfloor < x+1$, we have $$[d_0(i) - 1] \cdot 2^{s-1}  < d_1(i) < [d_0(i) + 1] \cdot 2^{s-1}.$$ Also, by construction, $d_0(i)$ is a multiple of $6$. Therefore, we have $$d_0(i) > 0 \implies d_1(i) > 5 \cdot 2^{s-1}$$ and \begin{align*} d_0(i) > d_0(j) \implies d_1(i) - d_1(j)> 4 \cdot 2^{s-1}.\end{align*} 

\medskip

Now, we construct the desired $\pi$-wave. The idea is to add two elements to $\mathcal{S}_1$ to create the two smallest consecutive differences. Let $\ell = \pi^{-1}(1)$ be the index of $1$ in $\pi$, and $r = \pi^{-1}(2)$ be the index of $2$ in $\pi$. Note that $\ell \leq r-2$. By \Cref{profilepick}, there exists some $c_\ell \in \{1,2\}$ for which $$x_{a_\ell + c_\ell} - x_{a_\ell + c_{\ell}-1} \leq \frac{x_{a_{l} + 3} - x_{a_{l}}}{2} < 2^{s-1},$$ where the final inequality holds as $x_{a_\ell} \in T_s$. Let $u_1 = x_{a_\ell + c_\ell - 1}$ and $u_2 = x_{a_\ell + c_{\ell}}$. 

Next, let $v = x_{a_{r-1} + 3}$. Note that since $x_{a_{r-1}} \in T_s$, we have $ v-z_{r-1} = v - x_{a_{r-1}} \in [2^{s-1}, 2^s)$. Then consider the sequence $$\mathcal{S} = (z_1, \ldots, z_{\ell-1}, u_1, u_2, z_{\ell+1}, \ldots, z_{r-1}, v, z_{r}, \ldots, z_{k-1}).$$ We claim that $\mathcal{S}$ is a $\pi$-wave. 

Intuitively, this is true as $u_2 - u_1$ and $v - z_{r-1}$ are the smallest consecutive differences, and the insertions of $u_1, u_2, v$ and deletion of $z_\ell$ from $\mathcal{S}_1$ do not significantly affect any other consecutive differences. More rigorously, suppose that $\pi(i) > \pi(j)$. Let $\mathcal{S}(i)$ be the $i$-th element of $\mathcal{S}$ and let $d(i) = \mathcal{S}(i+1) - \mathcal{S}(i)$ be the consecutive differences of $\mathcal{S}$. It suffices to show that $d(i) > d(j)$. We then have two cases.

\medskip

\noindent \textbf{Case 1: $i = r$.} Since $\pi(i) > \pi(j)$, we must have $j = \ell$. Then we have $$d(i) = v - z_{r-1} \geq 2^{s-1} > u_2 - u_1 = d(j).$$

\noindent \textbf{Case 2: $i \neq r$.} Since $\pi(i) > \pi(j)$, we must also have $i \neq \ell$. Let $i'$ be such that $i' = i$ if $i \leq \ell-1$, $i' = i-1$ if $\ell + 1 \leq i \leq r-1$, and $i' = i-2$ if $r+1 \leq i \leq k$. This is done so that for $i \neq \ell-1 ,\ell+1, r+1$, we have 
$$ d(i) = z_{i' +1} - z_{i'} = d_1(i');$$ for $i = l-1$, we have 
$$d(i) = u_1 - z_{\ell-1} = d_1(\ell-1) + (u_1 - z_{\ell}) \in \left[d_1(i'), d_1(i') + 2^s\right);$$ for $i = \ell + 1$, we have 
$$d(i) = z_{\ell+1} - u_2 = d_1(\ell) - (u_2 - z_\ell)\in \left(d_1(i') - 2^s, d_1(i')\right];$$ and for $i = r+1$, we have
$$d(i) = z_{r} - v = d_1(r-1) - (v - z_{r-1})\in \left(d_1(i') - 2^s, d_1(i')\right].$$
Hence, we have $$d_1(i') - 2^s < d(i) < d_1(i') + 2^s.$$

If $j \in \{\ell, r\}$, then $d(j) < 2^s$ by the construction of $u_1, u_2, v$. Therefore, we have $$d(i) > d_1(i') - 2^s > 5 \cdot 2^{s-1} - 2^s = 3 \cdot 2^s > d(j).$$
Otherwise, define $j'$ analogously to $i'$. Then $d(j)$ satisfies analogous inequalities to $d(i)$. Therefore, we have $$d(i) - d(j) > d_1(i') - d_1(j') - 2^{s+1} > 4 \cdot 2^{s-1} - 2^{s+1} = 0.$$
Therefore, $d(i) > d(j)$ in both cases. We conclude that $\mathcal{S}$ is a $\pi$-wave. 

Thus, any initial set $S$ satisfying the conditions must contain a $\pi$-wave $\mathcal{S}$. Hence by definition, $g(\pi, n) \leq 30(\log_2 n) \cdot g(\pi', n)$. This completes the proof.
\end{proof}
\begin{remark}\label{3hard}
    Unfortunately, the above proof does not easily extend to yield an analogous statement when removing more than two numbers from the permutation. In particular, suppose $\pi'$ is the result of removing $1$, $2$, and $3$ from $\pi$ and then normalizing. The above proof method then breaks down when attempting to construct the $\pi$-wave $\mathcal{S}$ from the $\pi'$-wave $\mathcal{S}_1$. Indeed, one would need to insert elements $u_1, u_2, u_3 \in S$ into $S_1$ to create the consecutive differences $u_1 - z_{a_1}$, $u_2 - z_{a_2}$, and $u_3 - z_{a_3}$ corresponding to $1$, $2$, and $3$, respectively. However, given no further conditions on the $z_i$, this is not always possible! We are thus unable to prove stronger recursive bounds.
\end{remark}
The complete classification of permutations for which the bound of \Cref{maincor} is tight, given by \Cref{equality}, now follows.

\begin{proof}[Proof of \Cref{equality}]
    The statement (1) follows from \Cref{peakseq}. We will prove the statement (2) for the density problem; the statement for the coloring problem then follows from \Cref{gtoP}. Proceed via induction on the length $k$ of $\pi$. The base cases are $\pi = 132$ and $\pi = 231$. In both cases, let $\pi' = 1$; then by \Cref{stronger} we have $$g(\pi, n) \leq 42(\log_2 n) \cdot g(\pi', n) = O\left(\log n)\right),$$ as desired. Now suppose the statement holds for any permutation of length at most $k-1$. Let $\pi \in S_k$ and suppose $\pi$ has some peak. We have two cases.

\medskip

\noindent \textbf{Case 1: $1$ and $2$ are adjacent.} Let $\pi' \in S_{k-1}$ be the permutation obtained by removing $1$ and then normalizing. Then we claim that $\pi'$ contains a peak. Indeed, let $\ell = \pi(i)$ be a peak in $\pi$. Note that we must have $\ell \geq 3$. If its neighbors $\pi(i-1), \pi(i+1) \neq 1$, then the neighbors of $\ell$ in $\pi'$ remain the same, hence $\ell$ remains a peak. Otherwise, suppose (without loss of generality) that $\pi(i-1) = 1$. Then since $\ell \neq 2$, its neighbors in $\pi'$ are $\pi(i+1)$ and $2$, hence $\ell$ remains a peak. Hence, by \Cref{main} and the inductive hypothesis, we have $$g(\pi, n) \leq 30(\log_2 n) \cdot g(\pi', n) = O\left((\log n)^{k-1}\right).$$

\noindent \textbf{Case 2: $1$ and $2$ are not adjacent.} Let $\pi' \in S_{k-2}$ be the permutation obtained by removing $1$ and $2$ and then normalizing. By \Cref{stronger} and \Cref{maincor}, we have $$g(\pi, n) \leq 42(\log_2 n) \cdot g(\pi', n) = O\left((\log n)^{k-1}\right).$$ This completes the inductive step in both cases and the proof.
\end{proof}

\begin{remark}\label{eqhard}
    Although we have shown a tight bound for a class of permutations, due to \Cref{3hard}, we are unable to achieve tight bounds in all cases. In particular, even if $\pi$ contains many peaks, this property is not easily exploited in general. However, in some cases, we are able to attain tight bounds by finding better constructions, as we will see in the next section.
\end{remark}
\section{Decomposable and Layered Permutations}\label{layered}

We now examine the cases in which \Cref{main} is not tight, proving the tight bound of \Cref{deccor} for layered permutations. First, we use \Cref{stronger} to show the upper bound of \Cref{deccor}.

\begin{lemma}\label{decbound}
    Let $\pi \in S_k$ be a layered permutation. Let $\ell$ be the number of non-final layers of $\pi$ of size at least $2$. Then $$g(\pi, n) = O\left((\log_2 n)^{k-\ell-1}\right).$$
\end{lemma}
   
\begin{proof}
     Induct on $k$. The base cases of $k = 1,2$ are resolved by \Cref{maincor}. Now suppose the statement holds for all permutations of length at most $k-1$. Let $\pi = \pi' \ominus (12\cdots t) \in S_k$. We have three cases.
    \medskip
    
\noindent \textbf{Case 1: $t \geq 2$.} Then let $\pi'' = \pi' \ominus (12\cdots (t-1))$. By \Cref{main} and the inductive hypothesis, $$g(\pi, n) \leq 30(\log_2 n) \cdot g(\pi'', n) = O\left((\log_2 n)^{1 + [(k-1) - l - 1]}\right) = O\left((\log_2 n)^{k-l-1}\right).$$
    
\noindent \textbf{Case 2: $t =1$ and $\pi'$ ends with $1$.} Then the final layer of $\pi'$ has length $1$. Therefore, by \Cref{main} and the inductive hypothesis, $$g(\pi, n) \leq 30(\log_2 n) \cdot g(\pi', n) = O\left((\log_2 n)^{1 + [(k-1) - l - 1]}\right) = O\left((\log_2 n)^{k-l-1}\right).$$ 

\noindent \textbf{Case 3: $t=1$ and $\pi'$ does not end with $1$}. Then the final layer of $\pi'$ has length at least $2$. Let $\pi''$ be the permutation obtained by removing $1$ and $2$ from $\pi'$ and normalizing. By \Cref{stronger} and the inductive hypothesis, $$g(\pi, n) \leq 42(\log_2 n) \cdot g(\pi'', n) = O\left((\log_2 n)^{1 + [(k-2) - (l-1) - 1]})\right) = O\left((\log_2 n)^{k-l-1}\right).$$

    This proves the inductive step in all cases, completing the proof.
\end{proof}

For the lower bound, we will use the following modification of $\pi$-waves. 

\begin{definition}
    Let $\pi \in S_k$. A \emph{weak-difference $\pi$-wave} is an increasing sequence of integers $(x_1, \ldots, x_{k+1})$ such that for all $1 \leq i,j \leq k$, we have $x_{i+1} - x_i \geq x_{j+1} - x_j$ if $\pi(i) > \pi(j)$.
\end{definition} 

Now, let $P_{\mathrm{weak}}(\pi, r)$ be the least $M$ such that any $r$-coloring of $[M]$ contains a weak-difference $\pi$-wave. We also need the following fact.

\begin{remark}\label{h}
    \Cref{ezconst} also holds for weak-difference $\pi$-waves. That is, the statement remains true if one replaces $P$ with $P_\mathrm{weak}$ everywhere. The proof is identical to that of \Cref{ezconst}.
\end{remark}

 We will also use the \emph{direct difference} of permutations, defined as follows.
\begin{definition}
    Let $\pi \in S_k$ and $\pi' \in S_\ell$. Then the \emph{direct difference} $\pi \ominus \pi'$ is defined to be the permutation $\sigma$ such that $\sigma(i) = \pi(i) + \ell$ for $i \leq k$ and $\sigma(i) = \pi'(i-k) $ for $i > k$. 
\end{definition}
We now have the following recursive lower bound on $P_{\mathrm{weak}}(\pi, r)$. 

\begin{lemma}\label{decconst}
    Let $\pi = \pi_L \ominus \pi_R \in S_k$, where $\pi_L \in S_\ell$ and $\pi_R \in S_{k - \ell}$. Suppose that $P_\mathrm{weak}(\pi_L, r) = \Omega\left(r(\log r)^{\ell_0}\right)$ and $P_\mathrm{weak}(\pi_R, r) = \Omega\left(r\log r)^{r_0}\right)$. Then $$P_\mathrm{weak}(\pi, r) = \Omega\left((\log n)^{\ell_0 + r_0}\right).$$ 
\end{lemma}
\begin{proof}
    For convenience, let $\hat{P}_{\mathrm{weak}}(\pi, r) = P_\mathrm{weak}(\pi, r) - 1$. Let $m = \left\lfloor \frac{\sqrt{r}}{10}\right\rfloor$ and set $$m_L = \hat{P}_\mathrm{weak}(\pi_L, m), \hspace{0.2cm} m_R = \hat{P}_\mathrm{weak}(\pi_R, m).$$ Let $C_L: [m_L] \rightarrow [m]$ be an $m$-coloring of $[m_L]$ without any monochromatic $\pi_L$-wave. Similarly, let $C_R: [m_R] \rightarrow [m]$ be an $m$-coloring of $[m_R]$ without any monochromatic $\pi_R$-wave. We will combine these colorings to form a $5m^2$-coloring of $[m_L\cdot m_R]$.  
    
    Indeed, let $S = [m] \times [m] \times [5]$. Then let $\mathcal{C}: [m_L \cdot m_R] \rightarrow S$ be the coloring defined as follows. For every $x \in [m_L \cdot m_R]$, write $$x = m_R(a-1) + \frac{m_R}{5}(b-1) + c,$$ where $a \in [m_L]$, $b \in [5]$, and $c \in [\frac{m_R}{5}]$. Then let $$\mathcal{C}(x) = \left(\mathcal{C}_L(a), \mathcal{C}_R(c), b\right).$$ We claim that $\mathcal{C}$ contains no monochromatic weak-difference $\pi$-wave. 
    
    Indeed, suppose otherwise that $$(x_1, \ldots, x_{k+1}) = \left(m_R(a_1-1) + \frac{m_R}{5}b_1 + c_1, \ldots, m_R(a_{k+1}-1) + \frac{m_R}{5}b_{k+1} + c_{k+1}\right)$$ is a monochromatic weak-difference $\pi$-wave. Then, all the $b_i$ must be equal; let $b_i = b$. Furthermore, let $\pi_L^{-1}(1) = t$. We now have two cases.
    
    %
\medskip

\noindent \textbf{Case 1: $a_{t+1} = a_{t}$.} Then $x_{t+1} - x_{t} < \frac{m_R}{5}$. It follows that $x_{i+1} - x_i < \frac{m_R}{5}$ for all $i \geq \ell+1$. Therefore, by construction we know that $a_{i+1} = a_i$ for all $i \geq \ell+1$. But then $(c_{\ell+1}, \ldots, c_{k+1})$ is a weak-difference $\pi_R$-wave, a contradiction.

\medskip

\noindent \textbf{Case 2: $a_{t+1} > a_t$.} Then we know that $x_{t+1} - x_{t} >\frac{4m_R}{5}$. Therefore, we must have $a_{i+1} > a_i$ for all $i \leq \ell$. However, since $c_i < \frac{m_R}{5}$, we know that for any $i,j \leq \ell$ such that $x_{i+1} - x_i \geq x_{j+1} - x_j$, we must also have $a_{i+1} - a_i \geq a_{j+1} - a_j$. Therefore, $(a_1, \ldots, a_{\ell+1})$ is a weak-difference $\pi_L$-wave, a contradiction.

Therefore, $\mathcal{C}$ must in fact contain no monochromatic weak-difference $\pi$-wave. It follows from the choice of $m, m_L, m_R$ that $$P_\mathrm{weak}(\pi, r) >\hat{P}_\mathrm{weak}(\pi, r) \geq m_L \cdot m_R = \Omega\left(r (\log_2 r)^{l_0 + r_0}\right),$$ as desired.
\end{proof}
    
This allows us to prove the lower bound in \Cref{deccor}.

\begin{lemma}\label{decconstcor}
    Let $\pi \in S_k$ be a layered permutation. Let $\ell$ be the number of non-final layers of $\pi$ of size at least $2$. Then $$P(\pi, r) \geq P_\mathrm{weak}(\pi, r) = \Omega\left(r(\log_2 r)^{k-\ell-1}\right).$$
\end{lemma}
\begin{proof}
    It is clear that $P(\pi, r) \geq P_\mathrm{weak}(\pi, r)$ as any $\pi$-wave is also a weak-difference $\pi$-wave. So it suffices to prove the latter inequality. We induct on the number of layers of $\pi$. The base case of a one-layer permutation holds by \Cref{maincor}. Now suppose the statement holds for all permutations with at most $t-1$ layers. Let $\pi$ have $t$ layers and suppose $\pi = \pi_1 \ominus  \pi'$, where $\pi_1 \in S_{k_1}$ is the identity permutation. By repeatedly applying \Cref{h}, we have $P_\mathrm{weak}(\pi_1, n) = \Omega\left((\log_2 n)^{k_1 - 1}\right)$. We have two cases.
    \medskip
    
\noindent \textbf{Case 1: $k_1 \geq 2$.} Then $\pi'$ has $\ell-1$ non-final layers of length at least $2$. Then by \Cref{decconst} and the inductive hypothesis, we have $$P_\mathrm{weak}(\pi, n) = \Omega\left((\log_2 n)^{(k_1 - 1) + [(k-k_1) - (l-1) - 1]}\right) = \Omega\left((\log_2 n)^{k - l-1}\right).$$ 
    
\noindent \textbf{Case 2: $k_1 = 1$.} Then $\pi'$ has $\ell$ non-final layers of length at least $2$. Then by \Cref{h}, $$P_\mathrm{weak}(\pi, n) = \Omega\left((\log_2 n)^{ 1 + [(k-1) - l - 1]}\right) = \Omega\left(\log_2 n)^{k-l-1}\right).$$ This completes the inductive step in both cases, as desired. 
\end{proof}

\Cref{deccor} follows immediately.
\begin{proof}[Proof of \Cref{deccor}]
   The upper bound follows from \Cref{decbound} and \Cref{gtoP}, and the lower bound follows from \Cref{decconstcor} and \Cref{gtoP}.
\end{proof}

\section{Relating the Density and Coloring Problems: \Cref{gtoP}}\label{cors}
Lastly, we prove \Cref{gtoP}, which we have used throughout to obtain bounds on $P(\pi, r)$ from bounds on $g(\pi, n)$ and vice versa. 

\begin{proof}[Proof of \Cref{gtoP}]
We first prove item (1). Suppose that $g(\pi, n) \leq C(\log n)^t$ for some $C > 0$. Let $M_r = \left\lceil 2^{t+1}C\cdot r(\log r)^t \right\rceil$ and consider any $r$-coloring of $[M_r]$. By the Pigeonhole Principle, some color is used at least $2^{t+1}C\cdot (\log r)^t$ times. Then note that $$2^{t+1}C \cdot (\log r)^t> 2C\left[\log r + \log \left(t\log r\right)\right]^t=  2C\left[\log \left(r(\log r)^t\right)\right]^t> C(\log M_r)^t$$ for sufficiently large $r$. Hence, by the assumption there must exist some $\pi$-wave of this color, implying that $P(\pi, M_r) \leq 2^{t+1}\cdot r(\log r)^t$ for sufficiently large $r$. Since $P(\pi, r)$ is increasing in $r$, it follows that $P(\pi, r) = O\left(r(\log r)^t\right)$, as desired.

We now prove item (2). Suppose that $P(\pi, r) \geq Cr(\log r)^t$ for some $C > 0$. Let $M_r = \lfloor Cr(\log r)^t\rfloor$ and consider an $r$-coloring of $[M_r]$ which does not contain any monochromatic $\pi$-wave. By Pigeonhole, some color $c$ is used at least $\frac{M_r}{r}$ times. Let $S$ be the subset of $[M_r]$ with color $c$. Then $S$ is $\pi$-wave-free and $|S| \geq C(\log r)^t - 1.$ Hence $g(\pi, M_r) \geq C(\log r)^t - 1$. Since $g(\pi, n)$ is increasing in $n$, it follows that $g(\pi, n) = \Omega\left((\log r)^t\right)$, as desired.
\end{proof}

\section{Conclusion and Further Directions}\label{concl}
We conclude with a few notes on further directions. \Cref{main} and \Cref{maincor} show that the asymptotic for $g(\pi, n)$ in general is at most polylogarithmic. We therefore ask for the exact asymptotic. 

\begin{question}
    What is the exact asymptotic behavior of $g(\pi, n)$ for fixed $\pi$ and large $n$? In particular, is it the case that for each $\pi$, there is some $e_\pi$ such that $g(\pi, n) = \Theta\left((\log n)^{e_\pi}\right)$?
\end{question}

We suspect that such $e_\pi$ should exist, but our methods are thus far unable to immediately yield the exact value of the exponent for all $\pi$. Lastly, note that the upper bounds given in \Cref{maincor} and \Cref{deccor} are both of the form $g(\pi, n) =O\left((\log n)^{e_\pi}\right)$, where the exponent $e_\pi \geq \frac{k}{3}$ is at least linear in the length $k$ of $\pi$. We therefore end by asking whether it is true that the exponent can be sublinear. 

\begin{question}
    Is it true that for any constant $C$, there is some $\pi \in S_k$, $k > 1$ for which $g(\pi, n) = O\left((\log n)^\frac{k}{C}\right)$? 
\end{question}

Again, we suspect that this is true. In particular, though the proof of \Cref{stronger} does not appear to extend to stronger recursive bounds when removing more than two elements of $\pi$, it suggests that if $\pi$ has very few pairs of elements which are adjacent and consecutive, such a bound may be attainable.

\end{document}